\documentclass[11pt]{amsart}
\usepackage{amssymb}
\usepackage{color}
\newtheorem{theorem}{Theorem}[section]

\newtheorem{lemma}[theorem]{Lemma}
\newtheorem{proposition}[theorem]{Proposition}
\theoremstyle{definition}
\newtheorem{definition}[theorem]{Definition}
\theoremstyle{remark}
\newtheorem{remark}[theorem]{Remark}

\numberwithin{equation}{section}

\begin{document}

\title[the Bernstein inequality for slice regular polynomials]{the Bernstein  inequality for slice regular polynomials}
\thanks{This work was supported by the National Natural Science Foundation of China (No. 11801125) and the Fundamental
Research Funds for the Central Universities (Nos. JZ2019HGTB0054 and JZ2018HGBZ0118).}
\author[Z. H. Xu]{Zhenghua Xu}
\address{Zhenghua Xu, School of Mathematics, HeFei University of Technology, Hefei 230601, China}
\email{zhxu$\symbol{64}$hfut.edu.cn}
\begin{abstract}
Due to the invalidation  of the Gauss-Lucas type result for quaternionic polynomials,    we first give in this paper an alternative proof  of  the  Bernstein  inequality in $L^{p} (1\leq p \leq+\infty)$  for slice regular polynomials by   the Fej\'er  kernel and the Minkowski inequality.   Secondly,  we extend a result of  Ankeny-Rivlin  to the quaternionic setting via the Hopf lemma.  By the way, some Turan inequalities  are  established for  slice regular polynomials.
\end{abstract}
\maketitle

\section{Introduction}
Let $\mathbb{P}_{n}(\mathbb{C})$ be the class of complex polynomials $P(z)=\sum^{n}_{j=0} a_{j}z^{j}$,   where $  a_{j}\in \mathbb{C}, j=0,1,\ldots,n.$ For $P \in \mathbb{P}_{n}(\mathbb{C})$, define
$$ \|P\|_{p}:=  \bigg( \frac{1}{2\pi}\int_{0}^{2\pi}|P( e^{i\theta})|^{p}d\theta\bigg)^{\frac{1}{p}}, \quad 0<p<+\infty,$$
and
$$ \|P\|_{\infty}:= \max_{|z|=1} |P(z)| = \max_{|z|\leq 1} |P(z)|. $$
For any $P\in \mathbb{P}_{n}(\mathbb{C})$, it holds that
\begin{eqnarray} \label{main-result-c-p}
\|P'\|_{p}\leq n\|P\|_{p}, \quad 0<p \leq +\infty.
\end{eqnarray}

The case $p =+\infty$ of (\ref{main-result-c-p}) is known as  the  Bernstein  inequality  which  is one  of the most powerful   tools  with many important applications in approximation theory. It was proved by  Zygmund \cite{Zygmund} with the aid of  an interpolation of Riesz for  $1\leq p <+\infty$ and     by Arestov \cite{Arestov} by using the Jensen formula and the subharmonic function  for  $0 < p <1$.
For more extensions of the Bernstein  inequality for  complex polynomials, we refer to  \cite[Chapter 1]{Borwein}, \cite[Chapter 4]{Lorentz}, \cite[Chapter 2]{Marden}, \cite[Chapter 14]{Rahman} and \cite{Govil}.

Recently, the classical Bernstein  inequality has been extended to the quaternionic setting depending  heavily on the Gauss-Lucas type theorem and the structure  of zero sets for quaternionic  slice regular  polynomials   \cite{SS}.  However, the method  used in \cite{SS} is not valid  for the octonionic and Clifford algebraic  cases.  What is worse, the alleged Gauss-Lucas type theorem  given in \cite{Vlacci} for  quaternionic   polynomials  has not been  verified since there exits a deadly  mistake in \cite[Proposition 3.14]{Vlacci}.   Very recently, Ghiloni and Perotti have proved that the Gauss-Lucas type result  in \cite{Vlacci}  is  correct only for   quaternionic  polynomials of   degree $d=2$ \cite{GP}.   Hence, it is necessary to  prove the Bernstein  inequality for  quaternions in  other   ways.

In this paper, we shall first  give an alternative proof of the Bernstein  inequality for  quaternions in a more general version. Precisely speaking, we   establish the analogy of  (\ref{main-result-c-p}) for $1 \leq p \leq +\infty$   in the setting of  quaternions by applying the Fej\'er kernel  and the Minkowski inequality, which is also   applicable to both octonions   and Clifford algebras.    In fact, our method can be used to establish the case $1 \leq p \leq +\infty$  of (\ref{main-result-c-p}) for complex  polynomials with coefficients valued in real or  complex normed linear  spaces.

To state our results, let us recall   some preliminary definitions and notation   for  quaternions.

Denote by $\mathbb H$ the non-commutative, associative, real algebra of quaternions with standard basis $\{1,\,i,\,j, \,k\}$,  subject to the multiplication rules
$$i^2=j^2=k^2=ijk=-1.$$
 Every element $q=x_0+x_1i+x_2j+x_3k$ in $\mathbb H$ is composed by the \textit{real} part ${\rm{Re}}\, (q)=x_0$ and the \textit{imaginary} part ${\rm{Im}}\, (q)=x_1i+x_2j+x_3k$. The \textit{conjugate} of $q\in \mathbb H$ is then $\bar{q}={\rm{Re}}\, (q)-{\rm{Im}}\, (q)$ and its \textit{modulus} is defined by $|q|=\sqrt{q\overline{q}}= \sqrt{x_0^{2}+x_1^{2}+x_2^{2}+x_3^{2}}$. The  inverse of each nonzero element $q$ of $\mathbb{H}$ is given by $ q^{-1} =|q|^{-2}\overline{q}$.
 Every $q \in \mathbb H $ can be expressed as $q = x + yI$, where $x, y \in \mathbb R$ and $I= {\rm{Im}}\, (q) /|{\rm{Im}}\, (q)|$
 if ${\rm{Im}}\, q\neq 0$, otherwise we take $I$ arbitrarily such that $I^2=-1$.
Here $I$ is an element of the unit 2-sphere of purely imaginary quaternions,
$$\mathbb S=\big\{q \in \mathbb H:q^2 =-1\big\}.$$
For every $I \in \mathbb S $,  we  denote by $\mathbb C_I$ the plane $ \mathbb R \oplus I\mathbb R $, isomorphic to $ \mathbb C$, and  by $\mathbb{B}_I$ the intersection $ \mathbb{B}\cap \mathbb C_I $, where $\mathbb{B}=\{q\in \mathbb{H}: |q|<1\}$ denotes  the open unit ball of quaternions.

Throughout this paper, we consider the quaternionic   polynomials in the form of
 $$P(q)=\sum^{n}_{j=0} q^{j}a_{j}, \quad a_{j}\in \mathbb{H}, j=0,1,\ldots,n.$$
In fact, those polynomials are  (left) slice regular functions in the sense of  Gentili and  Struppa  \cite{GS2}.
For the slice regular polynomial $P(q)=\sum^{n}_{j=0} q^{j}a_{j}$,   its derivative  is defined as
$$P'(q)=\sum^{n}_{j=1} q^{j-1}ja_{j}.$$

Let $\mathbb{P}_{n}(\mathbb{H})$ be the class of slice regular    polynomials $P(q)=\sum^{n}_{j=0}q^{j}a_{j},    a_{j}\in \mathbb{H}, j=0,1,\ldots,n.$ Let $P_I$ be the   restriction of $P$ to $\mathbb{C}_I$. For $P \in \mathbb{P}_{n}(\mathbb{H})$,  define
$$ \|P\|_{p}:= \sup_{I\in \mathbb{S}} \bigg( \frac{1}{2\pi}\int_{0}^{2\pi}|P_I( e^{I\theta})|^{p}d\theta\bigg)^{\frac{1}{p}}, \quad 0<p<+\infty,$$
and
$$ \|P\|_{\infty}:= \max_{|q|=1} |P(q)|. $$

 With the above notation, we can now present the first main result  as follows.
\begin{theorem}\label{Bernstein}
For  $P \in \mathbb{P}_{n}(\mathbb{H})$ and $1 \leq p \leq +\infty$, we have
\begin{eqnarray} \label{main-result}
\|P'\|_{p}\leq n\|P\|_{p}.
\end{eqnarray}
The inequality is sharp. It becomes an equality for  $P(q)=q^{n}a_{n}$ with $a_{n}\in \mathbb{H}$.
\end{theorem}

\begin{remark}
For complex polynomials, the Bernstein  inequality in the statement of Theorem \ref{Bernstein}
is obtained, for example, by Corollary 14.6.4, p. 553 in the book \cite{Rahman}, for all $0<p\leq +\infty$.
The Bernstein inequality for quaternionic polynomials in the case when $0 < p < 1$ remains as an open question.
\end{remark}


The refinement of the Bernstein inequality was conjectured by Erd\"{o}s and proved by Lax \cite{Lax}  stating that
 \begin{eqnarray}\label{Lax}
\| P'\|_{\infty}   \leq \frac{n}{2}  \|P\|_{\infty}
 \end{eqnarray}
for those complex polynomials $P\in \mathbb{P}_{n}(\mathbb{C})$ which   have  no zero   in the open unit disk $\mathbb{D}=\{z\in \mathbb{C}: |z|<1\}$. See
\cite{Aziz} and references therein for other proofs of this inequality.

As pointed out in \cite[Theorem 3.1]{SS},  the Erd\"{o}s-Lax   inequality (\ref{Lax})   fails in the quaternionic setting in general and   holds true  for a subclass of the quaternionic polynomials   as follows.

\begin{proposition}\label{C}
Assume   that    $P\in \mathbb{P}_{n}(\mathbb{H})$   has no zero  in the ball $\mathbb{B}$ and  the zeros of $P$ are either spheres and/or real points and  in addition to possibly up to one isolated zero  $ \alpha \in \mathbb{H}\backslash \mathbb{R}$ that has multiplicity $1$.
Then
$$\|P'\|_{\infty}\leq \frac{n}{2}\|P\|_{\infty}.$$
 \end{proposition}

Using the convex  combination identity for slice regular functions \cite{Ren}, we shall  offer  a brief method to establish the following version which contains the result in Proposition \ref{C}.  Note that  the proof of Proposition \ref{Erdos} is also valid for the octonion   or Clifford algebra valued polynomials with all   coefficients  in a  complex plane. In fact, we can find that the hypothesis  in Proposition \ref{Erdos} that the slice regular polynomial preserves one slice can not be omitted in general as shown by the example $P(q)=q^{2}-q(i+j)+k$.

\begin{proposition}\label{Erdos}
Let  $P\in \mathbb{P}_{n}(\mathbb{H})$.  If  $P_{I}$ has no zero  in the open unit disk $\mathbb{B}_{I} \subset \mathbb{C}_{I}$ and  $P_{I}(\mathbb{C}_{I})\subset \mathbb{C}_{I}$ for some $I\in \mathbb{S}$, then
$$\|P'\|_{\infty}\leq \frac{n}{2}\|P\|_{\infty}.$$
\end{proposition}

For the slice regular polynomial $P(q)=\sum^{n}_{j=0} q^{j}a_{j}$,  denote  $\widetilde{P}(q)=\sum^{n}_{j=0} q^{j}a_{n-j}= q^{n} P(\frac{1}{q})$. It follows that
$$\max_{|q|= R}|P(q)| = R^{n}  \max_{|q|= R} | q^{-n}P(q)| = R^{n}  \max_{|q|= R} \big|\widetilde{P}(\frac{1}{q})\big|=
R^{n}  \max_{|q|= 1/R} |\widetilde{P}(q)|,  $$
As in the complex case, we have,  by    the  maximum modulus principle   for slice  regular functions \cite[Theorem 7.1]{GSS},
$$ \max_{|q|= 1/R} |\widetilde{P}(q)|   \leq    \max_{|q|= 1} |\widetilde{P}(q)|=   \max_{|q|= 1} | P (\frac{1}{q})| = \max_{|q|= 1} |P(q)|, \quad  R>1.$$
Hence,
\begin{eqnarray} \label{main-result-equavalent-1}
 \max_{|q|= R}|P(q)|  \leq R^{n}    \max_{|q|= 1} |P(q)|,  \quad  R>1.   \end{eqnarray}

We also establish the following result for  $0<R<1$.
\begin{theorem}\label{main-result-equavalent-11}
For  $P\in \mathbb{P}_{n}(\mathbb{H})$, we have
\begin{eqnarray} \label{main-result-equavalent-11-}
 \max_{|q|= R}|P(q)|  \geq R^{n}    \max_{|q|= 1} |P(q)|,  \quad  0<R<1.   \end{eqnarray}
\end{theorem}

As an application of the Erd\"{o}s-Lax   inequality, Ankeny and  Rivlin strengthened  the complex version of  (\ref{main-result-equavalent-1}) as   follows   \cite{Ankeny}.
\begin{theorem}\label{Ankeny-Theorem}
Let   $P\in \mathbb{P}_{n}(\mathbb{C})$  be  such that $ \|P\|_{\infty}= 1$. If $P$ has no zero in $\mathbb{D}$, then
$$ \max_{|z|=R} |P(z)|\leq \frac{1+R^{n}}{2}, \quad R>1,$$
with equality only for $P(z)= (\lambda + \mu z^{n}) /2$, where $ \lambda, \mu \in \partial \mathbb{D}$.
\end{theorem}

As pointed out in \cite{Ankeny}, the converse of Theorem \ref{Ankeny-Theorem} is false as  shown by the  example $P(z)=(z+ \frac{1}{2})(z+3)$. However, the following result for complex polynomials in the converse direction is valid. In fact, we can prove it in the quaternionic setting via the Hopf  lemma, instead of the  Gauss-Lucas   theorem.
\begin{theorem}\label{Ankeny-Converse}
Let $P\in \mathbb{P}_{n}(\mathbb{H})$ be  such that
$P(1)=\|P\|_{\infty}=1$, and
$$ \max_{|q|=R}|P(q)|\leq \frac{1+R^{n}}{2}, \quad 1<R<\delta+1,$$
where $\delta$ is any positive number.
Then $P$ does not have all its zeros within the open unit ball $\mathbb{B}$.
\end{theorem}

The remaining part of this paper is organized as follows. In Section 2, we  shall prove Theorem \ref{Bernstein} and then give some interesting  remarks.
For completeness, the proof of Proposition \ref{Erdos} is also given in Section 2. Finally, we establish Theorems    \ref{main-result-equavalent-11} and \ref{Ankeny-Converse} in Section 3.
It is worth mentioning that   Proposition \ref{Bloch} is  vital to prove  Theorems  \ref{main-result-equavalent-11} and  \ref{Ankeny-Converse} due to the non-commutativity of quaternions.

\section{Proof of   Theorem \ref{Bernstein}  and  Proposition \ref{Erdos}  }
To prove Theorem \ref{Bernstein}, we resort to the Fej\'{e}r kernel.
The Fej\'{e}r kernel is given by
$$F_{n}(x) = \frac{1}{n+1} \sum_{k=0}^{n}D_k(x),$$
where
$$D_k(x)=\sum_{s=-k}^k {\rm e}^{isx}$$
is the $k$-th order Dirichlet kernel.

It can also be written in a closed form as
$$F_n(x) = \frac{1}{n+1} \left(\frac{\sin \frac{(n+1) x}{2}}{\sin \frac{x}{2}}\right)^2 =
\frac{1}{n+1} \frac{1 - \cos(n+1)x}{1 - \cos x}.$$
The important feature of the Fej\'{e}r kernel is   the fact that it is non-negative.
The Fej\'{e}r kernel can also be expressed as
\begin{eqnarray} \label{eq:Fejer-kern-series}F_n(x)
=\sum_{j=-n}^{n}\left(1-\frac{|j|}{n+1}\right)e^{ijx}.\end{eqnarray}

Let $g:\mathbb{R}\rightarrow \mathbb{H}$ be continuous and $2\pi$-periodic. Consider its Fourier series
$$g(\theta)=\sum_{j=-\infty}^{+\infty}e^{ij\theta}c_{j}, \qquad c_{j}=\frac{1}{2\pi}\int_{0}^{2\pi}e^{-ij\theta}g(\theta)d\theta.$$
The $n$-th partial sum of the Fourier series is given by
$$s_{n}(\theta;g)=\sum_{j=-n}^{n}e^{ij\theta}c_{j}
=\frac{1}{2\pi}\int_{0}^{2\pi}D_{n}(\theta-\varphi)g(\varphi)d\varphi$$
and  the corresponding $n$-th Ces\`{a}ro sum has the expression
 \begin{eqnarray}\label{Cesaro}
\sigma_{n}(\theta;g)=\frac{1}{n+1}\sum_{j=0}^{n}s_{j}(\theta;g)
=\frac{1}{2\pi}\int_{0}^{2\pi}F_{n}(\theta-\varphi)g(\varphi)d\varphi.
 \end{eqnarray}
 From (\ref{eq:Fejer-kern-series}), we thus have
 \begin{eqnarray}\label{Cesaro-90} \sigma_{n}(\theta;g)=\sum^{n}_{j=-n}(1-\frac{|j|}{n+1})e^{ij\theta}c_{j}. \end{eqnarray}
 for any $g(\theta)=\sum_{j=-\infty}^{+\infty}e^{ij\theta}c_{j}$.

With above preliminaries, we come to give the proof of Theorem \ref{Bernstein}.

\begin{proof}[Proof of Theorem \ref{Bernstein}]  Starting from the  slice regular polynomial $P(q)=\sum^{n}_{j=0} q^{j}a_{j}$, for any fixed $I\in \mathbb{S}$, we introduce the  continuous and $2\pi$-periodic function $g:\mathbb{R}\rightarrow {\mathbb H}$ via
$$g(\theta)=e^{In\theta}P(e^{-I\theta}).$$
It is evident that
\begin{eqnarray}  \label{maxg}
 |g(\theta)|=|e^{In\theta}P(e^{-I\theta})|= |P(e^{-I\theta})|,
\end{eqnarray}
 and
 \begin{eqnarray}\label{associa}
 g(\theta)=\sum^{n}_{j=0} e^{Ij\theta}a_{n-j}.
 \end{eqnarray}

Notice that
\begin{eqnarray}\label{deriative}
\left.\frac{1}{n}q^{-(n-1)}P'(q)\right|_{q=e^{I\theta}}
=\sum^{n}_{j=0} \frac{j}{n}e^{(j-n)I\theta}a_{j}
=\sum^{n-1}_{j=0}(1-\frac{j}{n})e^{-Ij\theta}a_{n-j}.\end{eqnarray}
The right side is equal to $\sigma_{n-1}(-\theta;g)$ due to (\ref{Cesaro-90}).

From   equality (\ref{Cesaro}), it follows that
\begin{eqnarray}\label{der}
 \sigma_{n-1}(-\theta;g)
=\frac{1}{2\pi}\int_{0}^{2\pi}F_{n-1}(-\theta-\varphi)g(\varphi)d\varphi. \end{eqnarray}
Combing this with (\ref{deriative}), we have
\begin{eqnarray*}\label{condition1}
 \left.\frac{1}{n}|q^{-(n-1)}P'(q)|\right|_{q=e^{I\theta}}
 &\leq & \frac{1}{2\pi}\int_{0}^{2\pi}F_{n-1}(-\theta-\varphi)|g(\varphi)|d\varphi
 \\
 &=& \frac{1}{2\pi}\int_{0}^{2\pi}F_{n-1}( \theta+\varphi)|g(\varphi)|d\varphi
 \\
  &= & \frac{1}{2\pi}\int_{0}^{2\pi}F_{n-1}( \varphi)|g(\varphi-\theta)|d\varphi,
 \end{eqnarray*}
which implies that, by the Minkowski inequality for $1\leq p<+\infty$ and (\ref{maxg}),
\begin{eqnarray*}\label{condition23}
 \bigg(\int_{0}^{2\pi}\Big|\frac{P'(e^{I\theta})}{n}\Big|^{p} \frac{d\theta }{2\pi}\bigg)^{\frac{1}{p}}
 &\leq& \int_{0}^{2\pi}  \frac{d\varphi}{2\pi}\bigg(\int_{0}^{2\pi}F_{n-1}^{p}( \varphi)|g(\varphi-\theta)|^{p}\frac{d\theta }{2\pi}\bigg)^{\frac{1}{p}}
 \\
 &=& \int_{0}^{2\pi} F_{n-1}( \varphi) \frac{d\varphi}{2\pi}\bigg(\int_{0}^{2\pi}|g(-\theta)|^{p}\frac{d\theta }{2\pi}\bigg)^{\frac{1}{p}}
\\
  &=& \bigg(\int_{0}^{2\pi}|P(e^{I\theta})|^{p}\frac{d\theta }{2\pi}\bigg)^{\frac{1}{p}}  .
     \end{eqnarray*}
Hence,
$$\|P'\|_{p}\leq n\|P\|_{p}, \quad 1\leq p<+\infty. $$

From  (\ref{maxg}), (\ref{deriative}) and (\ref{der}), the case of  $p=+\infty$ can be easily obtained. The proof is complete.
\end{proof}

It is worth mentioning that the proof of Theorem \ref{Bernstein} in the case $p =+ \infty$ follows the method  of \cite[Theorem 14.1.1]{Rahman}.

Two useful remarks concerning Theorem \ref{Bernstein} are in order.

\begin{remark}\label{O-C}
Denote by $\mathbb{R}_{0,m}$ the real Clifford algebra  over $m$ imaginary units  $e_1,e_2,\ldots,e_m$ which satisfy  $e_ie_j+e_je_i=-2\delta_{ij}.$ An element $a$ in  $\mathbb{R}_{0,m}$  is denoted by  $a=\sum_{A}a_{A}e_{A},$ where  $a_A \in \mathbb{R},  A = h_{1} \ldots h_{r}, 1\leq h_{1}<\ldots<h_{r}\leq m, e_{A}=e_{h_{1}}\ldots e_{h_{r}}$ and $e_{\emptyset}=1$. The modulus of  $a$ is defined by $|a|= ({\sum_{A}|a_{A}|^{2}} )^{\frac{1}{2}}.$  Note that the equality $|ab|=|a| |b|$ does not  hold generally for   $a, b \in \mathbb{R}_{0,m}$ when  $m\geq 3$. However, (\ref{maxg}) still holds for Clifford algebras  due to the following.

\begin{lemma}$($\cite[Theorem 3.14]{Gurlebeck}$)$\label{modulus}
For   $a, b\in \mathbb{R}_{0,m}$ with  $b\overline{b}=|b|^{2}$,  we have $$|ab|=|a| |b|.$$
In particular,
$$|xb|=|x||b|$$
for  any paravector $x = x_{0}+ x_{1}e_{1} + \ldots + x_{m}e_{m} \in \mathbb R^{m+1}$.
\end{lemma}

Hence, taking the same process as in Theorem \ref{Bernstein}, we can get the following Bernstein  inequality in the Clifford algebra setting  by the maximum modulus principle for  slice monogenic functions \cite[Theorem 3.1]{RX}. See \cite{CSS} for the precise  definition  of  slice monogenic functions.
\begin{theorem}
Let  $P(x)=\sum^{n}_{j=0} x^{j}a_{j}: \mathbb{R}^{m+1}\rightarrow \mathbb{R}_{0,m}$  be  a (left) slice monogenic polynomial of degree $n$ with  Clifford algebraic coefficients $a_{n}\in \mathbb{R}_{0,m}$. Then
 \begin{eqnarray*}
\|P'\|\leq n\|P\|,
 \end{eqnarray*}
where the norm of $P$ is defined by $\|P\|= \max_{|x|\leq1}|P(x)|$.

Moreover,   equality  holds if and only if $P(x)=x^{n}a_{n}$ for some $a_{n}\in \mathbb{R}_{0,m}$.
\end{theorem}
\end{remark}

\begin{remark}\label{O-C-1}
In fact, the non-associative nature of octonions plays no role in  (\ref{associa}) and (\ref{deriative}) in the proof of Theorem \ref{Bernstein}  since Artin's theorem (cf. \cite{RDS})  implies that the subalgebra generated by two elements in octonions is associative. Hence
Theorem \ref{Bernstein} still holds for octonionic slice regular polynomials
in the sense of  Gentili and  Struppa in \cite{GS3}.

\end{remark}

Now let us prove Proposition \ref{Erdos}.

\begin{proof}[Proof of Proposition \ref{Erdos}]
Let $P=\sum_{j=0}^{n}q^{j}a_{j}$ be the polynomial as described in the proposition.
Notice that $P_{I}(\mathbb{C}_{I})\subset\mathbb{C}_{I}$ for some $I\in \mathbb{S}$,
equivalently,  $a_{j}\in \mathbb{C}_{I}$ for   $j=0,1,\ldots,n$, which implies that $P'_{I}(\mathbb{C}_{I})\subset \mathbb{C}_{I}$. The classical  Erd\"{o}s-Lax   inequality  applied to $P_{I}$ yields that
 \begin{eqnarray}\label{E-Lax}
\|P_{I}'\|\leq \frac{n}{2}\|P_{I}\|,
 \end{eqnarray}
where $ \|P_{I}\|=\max_{q\in \partial\mathbb{B}_I}|P_{I}(q)| $.

Using  the convex combination identity for
slice regular functions  \cite{Ren}, it holds that
 \begin{eqnarray}\label{convex}
 |P'(\alpha+\beta J)|^{2}=\frac{1+\langle J,I \rangle}{2} |P'(\alpha+\beta I)|^{2}+\frac{1-\langle J,I \rangle}{2}|P'(\alpha-\beta I)|^{2},  \end{eqnarray}
where $\langle \cdot , \cdot \rangle$ is Euclidean inner product in $\mathbb{R}^{3}$.

From (\ref{E-Lax}) and  (\ref{convex}), we have
$$|P'(\alpha+\beta J)|\leq \frac{n}{2}\|P_{I}\|\leq\frac{n}{2}\|P\|_{\infty}, \quad \forall \alpha+\beta J\in\mathbb{ B}_{J}, \quad \forall J\in \mathbb{S},$$
as desired.
\end{proof}

\section{Proof  of Theorems  \ref{main-result-equavalent-11} and  \ref{Ankeny-Converse}}
Now we recall some necessary definitions and properties from \cite{CSSN,Lam,GSS} in order to prove  Theorem \ref{Ankeny-Converse}.
\begin{definition}Let $f$, $g:\mathbb{B} \rightarrow \mathbb H$ be two slice regular functions  of the form $$f(q)=\sum\limits_{n=0}^{\infty}q^na_n,\qquad g(q)=\sum\limits_{n=0}^{\infty}q^nb_n.$$ The regular product (or $\ast$-product) of $f$ and $g$ is the slice regular function defined by$$f\ast g(q)=\sum\limits_{n=0}^{\infty}q^n\bigg(\sum\limits_{k=0}^n a_kb_{n-k}\bigg).$$\end{definition}

Notice that the $\ast$-product is associative and is not, in general, commutative. Its connection with the usual pointwise product is clarified by the following result.

\begin{proposition} \label{product}
Let $f$ and $g$ be slice  regular on $\mathbb{B} $. Then for all $q\in \mathbb{B} $,
$$f\ast g(q)=\left\{\begin{array}{lll}f(q)g(f(q)^{-1}qf(q)) \qquad &\mathrm {if}  \  f(q)\neq 0;
\\      0\qquad  &\mathrm {if}  \ f(q)=0.\end{array}\right.$$
\end{proposition}

We remark that if $q=x+yI$ and $f(q)\neq0$, then $f(q)^{-1}qf(q)$ has the same modulus and same real part as $q$. Therefore $f(q)^{-1}qf(q)$ lies in the same 2-sphere $x+y\mathbb S$ as $q$. Notice that a zero $x_0+y_0I$ of the function $g$ is not necessarily a zero of $f\ast g$, but some element on the same sphere $x_0+y_0\mathbb S$ does. In particular, a real zero of $g$ is still a zero of $f\ast g$. To present a characterization of the structure of the zero set of a regular function $f$, we  introduce  the \emph{regular conjugate} of $f$ $$f^c(q)=\sum\limits_{n=0}^{\infty}q^n\bar{a}_n,$$ and the \emph{symmetrization} of $f$   $$f^s(q)=f\ast f^c(q)=f^c\ast f(q)=\sum\limits_{n=0}^{\infty}q^n\bigg(\sum\limits_{k=0}^n a_k\bar{a}_{n-k}\bigg).$$

\begin{proposition}\label{zero}  Let $f$ be a slice regular function on   $B(0, R)=\{q\in \mathbb{H} : |q|<R \}$ and choose $S=x+y\mathbb{S} \subset B(0, R)$. The zeros of $f$ in $S$ are in one-to-one correspondence with those of $f^{c}$. Furthermore, $f^{s}$ vanishes identically on $S$ if and only if $f^{s}$ has a zero in $S$, if and only if $f$ has a zero in $S$ (if and only if $f^{c}$ has a zero in $S$).
\end{proposition}

\begin{remark}
 It is easy to see that $ f^c(\cdot)=\overline{f(\overline{\cdot})} $   for any complex polynomial  $f$. However, this result  does not hold  for slice regular polynomials.
\end{remark}

Hence, we  also need the following result to prove  Theorems  \ref{main-result-equavalent-11}  and \ref{Ankeny-Converse}.
\begin{proposition}\label{Bloch}{\rm{(\cite[Proposition 3.2]{Rocchetta})}}
Let f be a slice regular function on  $B(0, R)$. For any sphere of the form $x+y\mathbb{S}$ contained in $B(0, R)$, the following equality holds
$$\sup_{I\in \mathbb{S}}|f(x + yI)|= \sup_{I\in \mathbb{S}}|f^{c}(x + yI)|.$$
\end{proposition}

Let us  prove  Theorem \ref{main-result-equavalent-11}.
\begin{proof}[Proof of theorem \ref{main-result-equavalent-11}]
Assume   $P(q)=\sum^{n}_{j=0} q^{j}a_{j}$ and  denote  $Q(q)=\sum^{n}_{j=0} q^{j}\overline{a_{n-j}}= q^{n} P^{c}(\frac{1}{q})$.
For $r>0$, we have, by Proposition \ref{Bloch},
$$\max_{|q|= r}|Q(q)| = r^{n}  \max_{|q|= r}\big |  P^{c} ( \frac{1}{q} )\big| = r^{n}  \max_{|q|= r} \big|P(\frac{1}{q})\big|=
r^{n}  \max_{|q|= 1/r} |P(q)|. $$
Hence, by the maximum modulus principle for slice regular functions,
$$ \max_{|q|= R}|P(q)| =   R^{n}    \max_{|q|= 1/R} |Q(q)|\geq R^{n}    \max_{|q|= 1 } |Q(q)|= R^{n}    \max_{|q|= 1 } |P(q)| ,  \quad 0<R<1, $$
as desired.
\end{proof}

To establish the complex version of Theorem \ref{Ankeny-Converse}, Ankeny and Rivlin used  the following lemma in terms of  the Laguerre  Theorem. Here we give an  elementary proof for the sake of completeness.

\begin{lemma}\label{Ankeny} If
$$P(z)=(z-z_1)\cdots (z-z_n)$$
where $ z_m \in \mathbb{D}=\{z\in \mathbb{C} : |z|<1\}$ for $m=1,2, \ldots, n$, then
$$\Big|\frac{P'(a)}{P(a)}\Big|>\frac{n}{2}, \quad \forall \ a\in  \partial\mathbb{D}.$$
\end{lemma}

\begin{proof}
Let $a=e^{i\theta}$ for some $ \theta\in \mathbb{R}$. Then,
$$ \textrm{Re}\frac{e^{i\theta}}{e^{i\theta}-z_{m}}>\frac{1}{2}$$
 for $z_m \in \mathbb{D}$, $m=1,2, \ldots, n$. Hence,
$$\Big|\frac{P'(a)}{P(a)}\Big|\geq \textrm{Re} \big( e^{i\theta}\frac{P'(e^{i\theta})}{P(e^{i\theta})}\big)
=\sum_{m=1}^{n}\textrm{Re}\frac{e^{i\theta}}{e^{i\theta}-z_{m}}>\frac{n}{2}.$$
\end{proof}

Now we can give the proof  of Theorem \ref{Ankeny-Converse}.
\begin{proof}[Proof of theorem \ref{Ankeny-Converse}]
Let us consider  the  symmetrization $P^{s}$ of the quaternionic polynomial $P$.  From Proposition     \ref{product}, it holds that
\begin{eqnarray}\label{1}
  P^{s}(q)=\left\{\begin{array}{lll}P(q)P^{c}(P(q)^{-1}qP(q)) \qquad &\mathrm {if}  \  P(q)\neq 0;
\\      0\qquad  &\mathrm {if}  \ P(q)=0.\end{array}\right.
 \end{eqnarray}
Hence
 \begin{eqnarray*}
  \max_{|q|=R} |P^{s}(q)|
  &=&\max_{ \{q\in \mathbb{H} \ :\ |q|=R, P(q)\neq 0\}   } |P^{s}(q)|
  \\
 &=& \max_{ \{q\in \mathbb{H} \ :\ |q|=R, P(q)\neq 0\}   } |P(q) P^{c}( P(q)^{-1}qP(q))|
 \\
 &\leq & \max_{|q|=R} |P(q)| \max_{ \{q\in \mathbb{H} \ :\ |q|=R, P(q)\neq 0\}   } |P^{c}(P(q)^{-1}qP(q) )|
  \\
   &\leq & \max_{|q|=R} |P(q)| \max_{|q|=R} |P^{c}(q)|,
 \end{eqnarray*}
Combining this with Proposition \ref{Bloch}, we have
 $$\max_{|q|=R} |P^{s}(q)|
  \leq   \max_{|q|=R} |P(q)| \max_{|q|=R} |P^{c}(q)|
  = \max_{|q|=R} |P(q)|^{2},$$
which implies that, by assumption,
\begin{eqnarray}\label{2}
 \|P^{s}\|_{\infty}\leq \|P\|_{\infty}^{2}=1,
 \end{eqnarray}
and
 \begin{eqnarray}\label{max} \max_{|q|=R} |P^{s}(q)|\leq \Big(\frac{1+R^{n}}{2}\Big)^{2}\leq \frac{1+R^{2n}}{2}, \quad 1<R<\delta+1. \end{eqnarray}

The condition $P(1)=1$ implies obviously that   $P^{c}(1)=1$, and then,   by
(\ref{1}),
  \begin{eqnarray}\label{3}
 P^{s}(1)=1.
 \end{eqnarray}
From  (\ref{2}) and  (\ref{3}), we see that  $ \|P^{s}\|_{\infty} =P^{s}(1)=1$.
 If $P^{s}$ is constant, then the claim is trivial. Otherwise,
  the  Hopf  lemma \cite[Lemma 15.3.7]{Rudin}   shows that $(P^{s})^{'} (1)>0$.
Then, given any $\epsilon>0$, sufficiently small, we have
$$|P^{s}(1+\epsilon)-P^{s}(1)|=  P^{s}(1+\epsilon)-1 \leq \frac{1+(1+\epsilon)^{2n}}{2}-1.$$
Hence     \begin{eqnarray}\label{A}
 0<(P^{s})^{'} (1)\leq n.
 \end{eqnarray}

Suppose that $P$ has all its zeros in $\mathbb{B}$. From Proposition  \ref{zero}, we see that $P^{s}$ has all its zeros in $\mathbb{B}$. Noticing  that $P^{s}$ is   a quaternionic  polynomial with real coefficients, we have,  by (\ref{3}) and Lemma \ref{Ankeny},
$$|(P^{s})^{'} (1)| > n,$$
which contradicts inequality (\ref{A}).
Now the proof is complete.
\end{proof}

An inverse inequality of (\ref{Lax})   was proved by Turan \cite{Turan}, which says that
$$ \| P'\|_{\infty}  \geq \frac{n}{2}  \|P\|_{\infty}$$
for  the  complex polynomial   $P\in \mathbb{P}_{n}(\mathbb{C})$ with all zeros in $\overline{\mathbb{D}}$.

From the proof of  Theorem \ref{Ankeny-Converse}, we obtain the following Turan inequality.
\begin{proposition}\label{Turan}
If   $P\in \mathbb{P}_{n}(\mathbb{H})$   has all its zeros in $\overline{\mathbb{B}}$ and   $P(1)=\|P\|_{\infty}$, then
 $$\|P' \|_{\infty} \geq  \frac{n}{2}\|P\|_{\infty}.$$
\end{proposition}

 \begin{remark}
For complex polynomials,  the condition $P(1)=\|P\|_{\infty}$ in Proposition \ref{Turan}  is not essential.  Indeed,   we can consider the polynomial $Q(z)=P(z_{0}z)\overline{P(z_{0})} $ when $|P(z_{0})|=\|P\|_{\infty}$ for some  $z_{0}\in \partial \mathbb{D}$.   Furthermore, there are many  quaternionic polynomials satisfying  the  statement of Proposition \ref{Turan} such as $P(q)=(q^{2}+2rq+1)(q^{n-2}u+q^{n-3}u)$ for any $r\in [0,1]$, integer $n\geq 2$ and quaternion $u$.  However, the construction of a concrete example of quaternionic polynomial of degree $d\geq 3$ with not all coefficients in the same plane  and  satisfying the hypothesis  of Proposition \ref{Turan}  remains as an open question.

\end{remark}

\begin{proposition}
If  $P\in \mathbb{P}_{2}(\mathbb{H})$   has all its zeros in $\overline{\mathbb{B}}$, then
 $$\|P' \|_{\infty} \geq   \|P\|_{\infty}.$$
\end{proposition}
\begin{proof}
By assumption,  the quaternionic polynomial $P\in \mathbb{P}_{2}(\mathbb{H})$ can be described as
$$P(q)=(q-\alpha)*(q-\beta)= q^{2}-q(\alpha+\beta)+ \alpha \beta,$$
 for some  $\alpha, \beta \in \mathbb{H} $ with  $ |\alpha|, |\beta| \leq 1$. Then we obtain readily  $\|P\|_{\infty} \leq  2+|\alpha+\beta| $ and  $P'(q)=2q-(\alpha+\beta).$ Let us now show that
 $$\|P' \|_{\infty} =\max_{|q|=1}|P'(q)|=2+|\alpha+\beta|.$$
If $\alpha+\beta=0$, this result is trivial.  Otherwise, notice  that $\|P'\|\leq 2+|\alpha+\beta|$ and choose  $q=- \frac{ \alpha+\beta }{|\alpha+\beta|}$, as  desired.
\end{proof}

\textbf{Acknowledgemants}
 The main result of this work is part of the author's Ph.D. thesis.
 The author would like express his hearty thanks  to his advisor, Professor Guangbin Ren, for helpful discussions. In addition, the author is very grateful to Professor Irene Sabadini  for valuable communications and  the anonymous referees for  constructive suggestions  which improve   significantly the presentation of the paper.

\bibliographystyle{amsplain}

\vskip 10mm

\end{document}